\documentclass[11pt]{amsart}
\usepackage{a4wide}
\usepackage{cite}
\usepackage{amsmath, amsfonts, amssymb,graphicx}
\usepackage{mathrsfs,url}

\theoremstyle{plain}
    \newtheorem{thm}{Theorem}
    \newtheorem{lem}[thm]{Lemma}
    \newtheorem{prop}[thm]{Proposition}

    \newtheorem{problem}[thm]{Problem}
    \newtheorem{theorem}[thm]{Theorem}

\theoremstyle{definition}
    
    \newtheorem{defn}[thm]{Definition}

    \newtheorem{definition}[thm]{Definition}

\theoremstyle{remark}

\DeclareMathOperator{\tp}{tp}
\DeclareMathOperator{\typ}{tp}
\DeclareMathOperator{\id}{id}
\DeclareMathOperator{\Aut}{Aut}

\DeclareMathOperator{\End}{End}
\DeclareMathOperator{\Pol}{Pol}

\DeclareMathOperator{\Expr}{Expr}
\DeclareMathOperator{\Csp}{CSP}
\newcommand{\exprpp}{\Expr_{\operatorname{pp}}}
\newcommand{\exprep}{\Expr_{\operatorname{ep}}}
\newcommand{\exprex}{\Expr_{\operatorname{ex}}}
\newcommand{\exprfo}{\Expr_{\operatorname{fo}}}
\newcommand{\To}{\rightarrow}
\newcommand{\nin}{\notin}
\newcommand{\F}{\mathcal F}
\newcommand{\C}{\mathcal C}
\newcommand{\E}{\mathcal E}
\newcommand{\D}{\mathcal D}
\newcommand{\R}{\mathcal R}
\newcommand{\M}{\mathcal M}
\newcommand{\N}{\mathcal N}
\newcommand{\inv}{^{-1}}
\newcommand{\mult}{\times}
\newcommand{\ignore}[1]{}

\date{Version 5 -- March 4, 2012}

\author{Manuel Bodirsky}
    \address{Laboratoire d'Informatique  (LIX), CNRS UMR 7161\\
    \'{E}cole Polytechnique \\91128 Palaiseau\\
    France}
    \email{bodirsky@lix.polytechnique.fr}
    \urladdr{http://www.lix.polytechnique.fr/~bodirsky/}
\author{Michael Pinsker}
    \address{\'{E}quipe de Logique Math\'{e}matique\\ Universit\'{e} Diderot -- Paris 7\\
	UFR de Math\'{e}matiques\\
	75205 Paris Cedex 13, France}
    \email{marula@gmx.at}
    \urladdr{http://dmg.tuwien.ac.at/pinsker/}
\author{Todor Tsankov}
    \address{\'{E}quipe de Logique Math\'{e}matique\\ Universit\'{e} Diderot -- Paris 7\\
	UFR de Math\'{e}matiques\\
	75205 Paris Cedex 13, France}
    \email{todor@math.jussieu.fr}
    \urladdr{http://people.math.jussieu.fr/~todor/}

    \thanks{The research leading to these results has received funding from the European Research Council under the European Community's Seventh Framework Programme (FP7/2007-2013 Grant Agreement no. 257039).
The second author is grateful for support through an APART-fellowship of the Austrian Academy of Sciences.  }

\title[Decidability of definability]{Decidability of definability}

\begin{document}

\begin{abstract}
    For a fixed countably infinite structure $\Gamma$ with finite relational signature $\tau$,
    we study the following computational problem: input are quantifier-free $\tau$-formulas $\phi_0,\phi_1,\dots,\phi_n$
    that define relations $R_0,R_1,\dots,R_n$ over $\Gamma$. The question is whether the relation $R_0$ is primitive positive definable from $R_1,\ldots,R_n$, i.e., definable by a first-order formula
    that uses only relation symbols for $R_1, \dots, R_n$, equality, conjunctions, and existential quantification (disjunction, negation, and universal quantification are forbidden).

    We show decidability of this problem for
    all structures $\Gamma$ that have a first-order definition in an ordered homogeneous structure $\Delta$ with a finite relational signature whose age is a Ramsey class and determined by finitely many forbidden substructures. Examples of structures $\Gamma$ with this property are the order of the rationals, the random graph, the homogeneous universal poset, the random tournament, all homogeneous universal $C$-relations, and many more. We also obtain decidability of the problem when we replace primitive positive definability by existential positive, or existential definability.
    Our proof makes use of universal algebraic and model theoretic concepts, Ramsey theory, and a recent characterization of Ramsey classes in topological dynamics.
\end{abstract}

\maketitle


\section{Motivation and the Main Result}\label{sect:motivationresult}

When studying a countably infinite relational structure $\Theta$, we often wish to know what $\Theta$ can express by its relations; for example, which other structures it interprets or defines. Concentrating on the latter, it would be pleasant to have an oracle which, given two structures $\Theta_1, \Theta_2$ on the same domain, tells us whether they define one another. If all structures we are interested in have finite signature, this is the same as having an oracle which, given a structure $\Theta$ and a relation $R$ on the same domain, tells us whether $R$ can be defined from $\Theta$.

In this context, different notions of definability can be considered. The first notion that comes to mind is probably \emph{first-order definability}: an $n$-ary relation $R$ is first-order definable in $\Theta$ iff there is a first-order formula $\phi(x_1,\ldots,x_n)$ over the language for $\Theta$ such that for all $n$-tuples $a$ of elements in $\Theta$ we have $a\in R$ iff $\phi(a)$ holds. In some applications, however, other notions of definability, in particular syntactic restrictions of first-order definability, are useful. We will be concerned here with \emph{primitive positive definability}: a first-order formula is called \emph{primitive positive} iff it is of the form $\exists y_1\ldots \exists y_m.\; \psi$, where $\psi$ is a conjunction of atomic formulas; and an $n$-ary relation $R$ is primitive positive definable over $\Theta$ iff it is first-order definable in
 $\Theta$ by means of a primitive positive formula $\phi(x_1,\ldots,x_n)$. 
 
Primitive positive definability is of importance in the study of the \emph{constraint satisfaction problem for $\Theta$}, denoted by $\Csp(\Theta)$, in theoretical computer science. In such a problem, the input consists of a primitive positive sentence $\psi$ (that is, a primitive positive formula without free variables), and the question is whether $\psi$ is true in $\Theta$. 
Primitive positive definability of relations in $\Theta$ 
is important in the study of $\Csp(\Theta)$ because the CSP for an expansion of $\Theta$ by relations that are primitive positive definable in $\Theta$ can be reduced (in linear time) to $\Csp(\Theta)$.  

We will present here conditions under which the oracle which is to tell us whether a relation $R$ has a primitive positive definition from a finite language structure $\Theta$ can be a computer, i.e., under which the problem is decidable. In order to make the problem suitable for an algorithm, we need a finite representation of the input of the problem, that is, the relation $R$ and the structure $\Theta$. Our approach is to fix a base structure $\Gamma$ with finite relational language, and to assume that both $R$ and $\Theta$ have a quantifier-free definition in $\Gamma$. We then represent $R$ and $\Theta$ as quantifier-free formulas over $\Gamma$. Therefore, the input of our problem are quantifier-free formulas $\phi_0,\ldots,\phi_n$ in the language of $\Gamma$, of which $\phi_0$ defines the relation $R$, and $\phi_1,\ldots,\phi_n$ define the relations $R_1,\ldots,R_n$ of $\Theta$; the question is whether there is a primitive positive definition of $\phi_0$ that uses only relation symbols for $R_1,\dots,R_n$. We denote this computational problem by $\exprpp(\Gamma)$.

An algorithm for primitive positive definability has theoretical and practical consequences in the study of the computational complexity of CPSs. On the practical side, it turns out that hardness of $\Csp(\Theta)$ can usually be shown by presenting primitive positive definitions of relations for which it is known that the CSP is hard. Therefore,
a procedure that decides primitive positive definability of a given relation is a useful tool to determine the computational complexity of CSPs. 

For the simplest of countable structures, namely the structure $(X;=)$ having no relations but equality, the decidability of $\exprpp(\Gamma)$ has been stated as an open problem in~\cite{BodChenPinsker}. 
We will show here decidability of $\exprpp(\Gamma)$ for a large class of structures $\Gamma$ which we will now define.

Let $\tau$ be a finite relational signature.
The \emph{age} of a $\tau$-structure $\Delta$ is the class
of all finite $\tau$-structures that embed into $\Delta$.
We say that a class $\mathcal C$ of finite $\tau$-structures, and similarly a structure with age $\mathcal C$, is
\begin{itemize}
\item \emph{finitely bounded}
(in the terminology of~\cite{MacphersonSurvey})
iff there exists a finite set of finite $\tau$-structures $\mathcal F$ such
that for all finite $\tau$-structures $A$ we have $A \in \mathcal C$ iff no structure from $\mathcal F$ embeds into $A$;
\item \emph{Ramsey} iff for all $k \geq 1$ and for all $H,P \in \mathcal C$ there exists $S\in\mathcal C$ such that $S \rightarrow (H)^P_k$, i.e., for all colorings of the copies of $P$ in $S$ with $k$ colors there exists a copy of $H$ in $S$ on which the coloring is constant (for background in Ramsey theory see~\cite{GrahamRothschildSpencer});
\item \emph{ordered} iff the signature $\tau$ contains a binary relation
that denotes a total order in every $A \in \mathcal C$.
\end{itemize}

A structure is called \emph{homogeneous} iff all isomorphisms between finite induced substructures\footnote{In this article, substructures are always meant to be \emph{induced}; see~\cite{Hodges}.} extend to automorphisms of the whole structure. 
A structure $\Gamma$ is called a \emph{reduct} of a structure $\Delta$ with the same
domain iff all relations in $\Gamma$ are first-order definable in $\Delta$. 
We will prove the following.

\begin{thm}\label{thm:main:pp}
	Let $\Delta$ be a structure which is ordered, homogeneous, Ramsey, finitely bounded, and has a finite relational signature.
	Then for any reduct $\Gamma$ of $\Delta$ with finite relational signature 
    the problem $\exprpp(\Gamma)$ is decidable.
\end{thm}

We remark that for \emph{finite} structures $\Gamma$ the problem $\exprpp(\Gamma)$ is in co-NEXPTIME (and in particular decidable). For the variant where the finite structure
$\Gamma$ is part of the input, the problem has recently shown to be also co-NEXPTIME-hard~\cite{Willard-cp10}.

Note that since $\Delta$ is homogeneous, it has \emph{quantifier
elimination}, i.e., every relation which is first-order definable in
$\Delta$ can be defined by a quantifier-free formula. Hence, choosing
$\Gamma=\Delta$, we see that our requirement for the relations in
$\exprpp(\Gamma)$ to be given by quantifier-free formulas does not
restrict the range of relations under consideration.

Examples of structures $\Delta$ that satisfy the assumptions
of Theorem~\ref{thm:main:pp} are $({\mathbb Q};<)$,
the Fra\"{i}ss\'{e} limit of ordered finite graphs (or tournaments~\cite{RamseyClasses}), the Fra\"{i}ss\'{e} limit of
finite partial orders with a linear extension~\cite{RamseyClasses}, and the homogeneous universal `naturally ordered' $C$-relations. (For definition and basic properties of $C$-relations, see~\cite{MR1388893}, in particular Theorem~14.7. The fact that the homogeneous universal naturally ordered $C$-relations have the Ramsey property follows from Theorem 4.3 in~\cite{Mil79}; an
 explicit and elementary verification of the Ramsey property for the binary
 branching case can be found in~\cite{BodirskyPiguet}.) CSPs of reducts of such structures are abundant
in particular for qualitative reasoning calculi in Artificial Intelligence. For instance, our result shows that it is decidable whether a given relation from Allen's Interval Algebra~\cite{Allen,RandomReducts} is primitive positive definable in a given fragment of Allen's Interval Algebra.

As mentioned above, for $\Gamma=(X; =)$, the decidability of $\exprpp(\Gamma)$ has been posed as an open problem in~\cite{BodChenPinsker}. Our results solve this problem, since $(X;=)$ is definable in $\Delta:=(\mathbb Q;<)$, which is  ordered, homogeneous, Ramsey, and finitely bounded: the Ramsey property for this structure follows from the classical Ramsey theorem, and the other properties are easily verified.

Using similar methods, decidability of the analogous problem for other syntactic restrictions of first-order logic can be shown in the same context. A formula is called \emph{existential} iff it is of the form $\exists y_1\ldots \exists y_m.\; \psi$, where $\psi$ is quantifier-free. It is called \emph{existential positive} iff it is existential and does not contain any negations. For a $\tau$-structure $\Gamma$, we denote by $\exprex(\Gamma)$ ($\exprep(\Gamma)$) the problem of deciding whether a given quantifier-free $\tau$-formula $\phi_0$ has an existential (existential positive) definition over the structure with the relations defined by given quantifier-free $\tau$-formulas $\phi_1,\ldots,\phi_n$ in $\Gamma$.

\begin{thm}\label{thm:main:exep}
	Let $\Delta$ be a structure which is ordered, homogeneous, Ramsey, finitely bounded, and has a finite relational signature.
	Then for any reduct $\Gamma$ of $\Delta$ with finite relational signature 
    the problems $\exprex(\Gamma)$ and $\exprep(\Gamma)$ are decidable.
\end{thm}

The assumptions on $\Delta$ in our theorems fall into two classes: the conditions of being ordered, homogeneous, Ramsey, and having finite relational signature imposed on $\Delta$ generally allow for a relatively good understanding (in a non-algorithmic sense) of the reducts of $\Delta$. The recent survey paper~\cite{BP-reductsRamsey} summarizes what we know about reducts of such structures -- their exciting feature is that many branches of mathematics, including model theory, combinatorics, universal algebra, and even topological dynamics are employed in their study, and indirectly also in our algorithm. The additional condition of being finitely bounded is needed to represent $\Delta$ algorithmically.

This paper is organized as follows. In Section~\ref{sect:henson} we show that the assumption of $\Delta$ being finitely bounded is necessary for our decidability result. We then turn to the proof of Theorems~\ref{thm:main:pp} and~\ref{thm:main:exep}: in Section~\ref{sect:preservation} we cite preservation theorems of the form ``$R$ is definable from $\Theta$  (in some syntactically restricted form of first-order logic) if and only if certain functions on the domain of $\Theta$  (which functions depends on the syntactic restriction) preserve $R$''. Section~\ref{sect:canonical} is devoted to the use of Ramsey theory in order to standardize functions that do not preserve $R$ -- if such functions exist. Our decision procedure, presented in Section~\ref{sect:algorithm}, then uses this standardization of functions and the preservation theorems to check whether or not $R$ is definable from $\Theta$. The paper ends with two sections containing further discussion and open problems.

\section{Undecidability of Definability}
\label{sect:henson}
This section demonstrates that the assumption in Theorem~\ref{thm:main:exep} of $\Delta$ being finitely bounded is necessary.
We use a class of homogeneous digraphs introduced by Henson~\cite{Henson}.
A \emph{tournament} is a directed graph without self-loops such that for
all pairs $x,y$ of distinct vertices exactly one of the pairs $(x,y)$, $(y,x)$
is an arc in the graph.  
For a set of finite directed graphs $\mathcal N$, we write $\text{Forb}(\mathcal N)$ for the class of all finite directed graphs that 
do not embed one of the structures from $\mathcal N$. 
For all sets $\mathcal N$ of finite tournaments there exists a countably infinite 
homogeneous directed graph $\Gamma$ with age $\text{Forb}(\mathcal N)$ (this can be shown by amalgamation, see~\cite{Hodges}).
Moreover, those properties characterize $\Gamma$ up to isomorphism.  
Henson specified an infinite set 
$\mathcal T$ of finite tournaments $\Lambda_1,\Lambda_2,\dots$ with the property that
$\Lambda_i$ does not embed into $\Lambda_j$ if $i \neq j$; the exact definition of this set
is not important in what follows. But note that for two distinct subsets ${\mathcal N}_1$ and ${\mathcal N}_2$ of $\mathcal T$ the two sets $\text{Forb}({\mathcal N}_1)$
and $\text{Forb}({\mathcal N}_2)$ are distinct as well, and so are the respective homogeneous digraphs with age $\text{Forb}({\mathcal N}_1)$ and $\text{Forb}({\mathcal N}_2)$. 
Since there
are $2^\omega$ many subsets 
of the infinite set $\mathcal T$, 
there are also that many distinct homogeneous directed graphs; 
they are often referred to as \emph{Henson digraphs}. 

\begin{prop}
There exists a ordered directed graph $\Delta$ which is homogeneous and Ramsey such that $\exprpp(\Delta)$ and $\exprep(\Delta)$ are undecidable.
\end{prop}
\begin{proof}
For any Henson digraph $\Gamma$, the class $\C$ of all expansions of the structures in the age of $\Gamma$ by a linear order is a Ramsey class; this can been shown by the partite method~\cite{NesetrilRoedlPartite}. Moreover, there exists a homogeneous ordered digraph $\Delta$ with age $\C$ (again by amalgamation, see~\cite{Hodges}),
and $\Gamma$ is a reduct of $\Delta$. 

We show that non-isomorphic Henson digraphs $\Gamma_1$ and $\Gamma_2$ 
have distinct $\exprpp$ problems.
In the following, let $E$ denote a binary relation symbol that we
use to denote the edge relation in graphs. 
In fact, we show the existence of a first-order formula $\phi_1$ over digraphs such that the input $\phi_0 := E(x,y)$ and $\phi_1$ is a yes-instance of $\exprpp(\Gamma_1)$ and a no-instance of $\exprpp(\Gamma_2)$, or vice-versa. Since there are uncountably many Henson digraphs, but only countably many
algorithms, this clearly shows the existence of  Henson digraphs $\Gamma$ such that
$\exprpp(\Gamma)$ is undecidable. This finishes the proof since $\Gamma$ is a reduct of an ordered homogeneous Ramsey structure $\Delta$, as we have seen above, and
$\exprpp(\Delta)$ must be undecidable as well. The same argument shows undecidability of $\exprep(\Delta)$. 

By the definition of $\Gamma_1$ and $\Gamma_2$, there exists a finite digraph $\Omega$ which
embeds into $\Gamma_1$ but not into $\Gamma_2$, or that embeds into $\Gamma_2$ but not into 
$\Gamma_1$.
Assume without loss of generality the former.
Let $s$ be the number of elements of $\Omega$, and denote its elements by $a_1,\dots,a_s$. 
Let $\psi$ be the formula with variables $x_1,\dots,x_s$ that has for distinct $i,j \leq s$
a conjunct $E(x_i,x_j)$ if $E(a_i,a_j)$ holds in $\Omega$, and a conjunct $\neg E(x_i,x_j) \wedge x_i \neq x_j$ otherwise. Let $\phi_1$ be the formula $\psi \wedge E(x_{s+1},x_{s+2})$. 

Let $D_1$ be the domain of $\Gamma_1$, 
and consider the relation $R_1 \subseteq (D_1)^{s+2}$ defined by $\phi_1$ in $\Gamma_1$.
Let $R$ be a relational symbol of arity $s+2$. 
Let $\Theta$ be the structure with signature $\{R\}$, domain $D_1$, and
where $R$ denotes the relation $R_1$. 
It is clear that 
$\exists x_1,\dots,x_s. \, R(x_1,\dots,x_s,x,y)$ is a primitive positive definition of $E(x,y)$ in $\Theta$. 

Now consider the relation $R_2$ defined by $\phi_1$ in $\Gamma_2$ over the domain $D_2$
of $\Gamma_2$. Since $\Omega$ does not embed into $\Gamma_2$, 
the precondition of $\phi_1$ is never satisfied, and
the relation $R_2$ is empty. Hence, the relation $E(x,y)$ is certainly not first-order (and in particular not primitive positive) definable in $(D_2;R_2)$.
\end{proof}

\section{Preservation Theorems}\label{sect:preservation}
Let $\Gamma$ be a reduct of a
homogeneous finitely bounded Ramsey structure $\Delta$ with finite relational signature.
Our algorithm for $\exprpp(\Gamma)$ is based on the fact that if $R$ is not definable from $\Theta$, then there exists a certain kind of function which violates $R$; in order to decide whether or not $R$ is definable, the algorithm thus searches for such a function. In this section, we shall formulate this fact in more detail.

A structure is called \emph{$\omega$-categorical} iff its first-order theory has exactly one countable model up to isomorphism. For an $n$-tuple $a$ of elements of a structure $\Delta$, the \emph{type} of $a$ is the set of all first-order formulas with $n$ free variables $x_1,\ldots,x_n$ that are satisfied by $a$. By a theorem of Ryll-Nardzewski (see for example the textbook~\cite{Hodges}), a structure is $\omega$-categorical iff it has only finitely many different types of $n$-tuples (called \emph{$n$-types}), for each $n\geq 1$. From this characterization it is straightforward to see that structures which are homogeneous and have a finite relational signature are $\omega$-categorical; in particular, this is true for the structure $\Delta$ of Theorems~\ref{thm:main:pp} and~\ref{thm:main:exep}. For an $n$-tuple $a$ of elements of a structure $\Delta$, the \emph{orbit} of $a$ is the set $\{\alpha(a):\alpha\in\Aut(\Delta)\}$, where $\Aut(\Delta)$ denotes the \emph{automorphism group} of $\Delta$. It is well-known that a structure is $\omega$-categorical iff it has for every $n\geq 1$ only finitely many orbits of $n$-tuples (called \emph{$n$-orbits}). Moreover, in $\omega$-categorical structures two $n$-tuples have the same type iff they have the same orbit (see again~\cite{Hodges}). 
In particular, every $n$-ary relation 
definable over an $\omega$-categorical structure is a finite union of orbits of $n$-tuples.

Clearly, when $\Theta$ is a reduct of a structure $\Delta$, then $\Aut(\Theta)\supseteq\Aut(\Delta)$. Hence, if $\Delta$ is $\omega$-categorical, then so is $\Theta$; therefore, all structures that appear in this paper are $\omega$-categorical.

If $R$ is an $m$-ary relation on a set $D$, and $f \colon D^n\To D$ is a finitary operation on $D$, then we say that $f$ \emph{preserves} $R$ iff $f(r_1,\ldots,r_n)$ (calculated componentwise) is in $R$ for all $m$-tuples $r_1,\ldots,r_n\in R$.
In other words, when $r_i = (r_i^1,\dots,r_i^m) \in R$ for all $i \leq n$, we require that $(f(r_1^1,\dots,r_n^1),\dots,f(r_1^m,\dots,r_n^m)) \in R$. 
 Otherwise, we say that $f$ \emph{violates} $R$. Observe that a permutation $\alpha$ acting on the domain of a structure $\Theta$ is an automorphism iff both $\alpha$ and its inverse preserve all relations of $\Theta$. An \emph{endomorphism} of a structure $\Theta$ with domain $D$ is a unary operation $f \colon D\To D$ which preserves all relations of $\Theta$. A \emph{self-embedding} of $\Theta$ is an injective unary operation $f \colon D\To D$ which preserves all relations of $\Theta$ and all complements of relations in $\Theta$. A \emph{polymorphism} of $\Theta$ is a finitary operation $f \colon D^n\To D$ which preserves all relations of $\Theta$.

We can now state the preservation theorem used by our algorithm. Statement~(1) is well-known in model theory and follows from the standard proof of the theorem of Ryll-Nardzewski. Items~(2) and~(3) are consequences of the Theorem of {\L}os--Tarski and the Homomorphism Preservation Theorem; for these theorems, see~\cite{Hodges}, for the (straightforward) proofs of statements~(2) and (3) see~\cite{RandomMinOps}. Item~(4) is due to Bodirsky and Ne\v{s}et\v{r}il~\cite{BodirskyNesetrilJLC}.

\begin{theorem}\label{thm:preservation}
    Let $\Theta$ be an $\omega$-categorical structure, and let $R$ be a relation on its domain.
    \begin{itemize}
        \item[(1)] $R$ has a first-order definition in $\Theta$ iff $R$ is preserved by all automorphisms of $\Theta$.
        \item[(2)] $R$ has an existential definition in $\Theta$ iff $R$ is preserved by all self-embeddings of $\Theta$.
        \item[(3)] $R$ has an existential positive definition in $\Theta$ iff $R$ is preserved by all endomorphisms of $\Theta$.
        \item[(4)] $R$ has an primitive positive definition in $\Theta$ iff $R$ is preserved by all polymorphisms of $\Theta$.
    \end{itemize}
\end{theorem}

\section{Standardizing Functions}\label{sect:canonical}

Theorem~\ref{thm:preservation} tells us that if a relation $R$ is not definable in an $\omega$-categorical structure $\Theta$, then this is witnessed by a some finitary function on the domain of $\Theta$; the kind of function depends on the notion of definability. In this section, we show that in the context of Theorems~\ref{thm:main:pp} and~\ref{thm:main:exep}, this is even witnessed by a function which shows a certain regular behavior, making the search for such an (infinite!) function accessible to algorithms. We start by defining what we mean by regular behavior.

\subsection{Canonicity}

\begin{defn}
    For a structure $\Delta$ and $n\geq 1$, we write $S_n^\Delta$ for the set of all $n$-types in $\Delta$. The cardinality of $S_n^\Delta$ is denoted by $o^\Delta(n)$. We write $S^\Delta:=\bigcup_{n\geq 1} S_n^\Delta$. For an $n$-tuple $a\in\Delta$, we write $\typ^\Delta(a)$ for the element of $S_n^\Delta$ corresponding to $a$. We drop the reference to the structure in this notation when the structure is clear from the context.
\end{defn}

\begin{definition}
    A \emph{type condition} between two structures $\Xi$ and $\Omega$ is a pair $(s,t)$, where $s\in S_n^{\Xi}$ and $t\in S_n^{\Omega}$ for the same $n\geq 1$. A function $f \colon \Xi\To \Omega$ \emph{satisfies} a type condition $(s,t)$ iff for all $n$-tuples $a=(a_1,\ldots,a_n)$ in $\Xi$ of type $s$, the $n$-tuple $f(a)=(f(a_1),\ldots,f(a_n))$ in $\Omega$ is of type $t$.

    A \emph{behavior} is a set of type conditions between two structures. A function \emph{has behavior $B$} iff it satisfies all the type conditions of the behavior $B$. For $n\geq 1$, a behavior $B$ is called \emph{$n$-complete} iff for all types $s\in S^{\Xi}_n$ there is a type $t\in S^{\Omega}_n$ such that $(s,t)\in B$. It is called \emph{complete} iff it is $n$-complete for all $n\geq 1$.

    A function $f \colon  \Xi \To \Omega$ is \emph{canonical} ($n$-canonical) iff it has a complete ($n$-complete) behavior. 
    
    For $F\subseteq \Xi$ we say that $f$  \emph{satisfies a type condition $(s,t)$ on $F$} iff  for all $n$-tuples $a=(a_1,\ldots,a_n)$ in $F$ of type $s$ (in $\Xi$, not in the substructure induced by $F$), the $n$-tuple $f(a)=(f(a_1),\ldots,f(a_n))$ in $\Omega$ is of type $t$. The notions of \emph{having a behavior on $F$} and of \emph{being canonical on $F$} are then defined naturally.
\end{definition}

Observe that a complete behavior is just a function from $S^{\Xi}$ to $S^{\Omega}$ which respects the sorts, i.e., $n$-types are sent to $n$-types. We remark that not every such function is necessarily the behavior of a canonical function from $\Xi$ to $\Omega$, but every canonical function from $\Xi$ to $\Omega$ does define a function from $S^{\Xi}$ to $S^{\Omega}$. A behavior is just a partial function from $S^{\Xi}$ to $S^{\Omega}$ respecting the sorts.

\begin{defn}
    For a relational structure $\Delta$, we write $n(\Delta)$ for the supremum of the arities of the relations of $\Delta$.
\end{defn}


Suppose that $n(\Xi)$ is finite and that $\Xi$ has \emph{quantifier
elimination}, i.e., every first-order formula in the language of $\Xi$
is equivalent to a quantifier-free formula over $\Xi$; this is in
particular the case for the structure $\Delta$ of
Theorems~\ref{thm:main:pp} and~\ref{thm:main:exep}, since homogeneity
implies quantifier elimination. Then the type of any tuple in $\Xi$ is
determined by the types of its subtuples of length ${n(\Xi)}$. If
moreover the same condition holds for $\Omega$ (in particular, if
$\Omega=\Xi$), and we set $n$ to be the maximum of $n(\Xi)$ and
$n(\Omega$), then a total function from $S^{\Xi}_{n}$ to
$S^{\Omega}_{n}$ automatically defines a total function from
$S^{\Xi}_{}$ to $S^{\Omega}_{}$. In other words, a function $f \colon
\Xi\To\Omega$ is canonical iff it is $n$-canonical. Note also that
$S^{\Xi}_k$ is finite for every $k\geq 1$ since $\Xi$ is
$\omega$-categorical (this follows if $\Xi$ has quantifier elimination
and finite relational signature, cf.~\cite{Hodges}). Therefore,
canonical functions can be represented by finite objects, namely by
functions from $S^{\Xi}_{n}$ to $S^{\Omega}_{n}$. Since $\Omega$ is
$\omega$-categorical as well, there are only finitely many functions
from $S^{\Xi}_{n}$ to $S^{\Omega}_{n}$, and hence there exist only
finitely many complete behaviors between $\Xi$ and $\Omega$, allowing
to check all of them in an algorithm. Roughly, our goal in the
following is to prove that functions witnessing that a relation $R$ is
not definable in $\Theta$ can be assumed to be canonical; it will turn out that this is almost true.

\subsection{Calling Ramsey}

\begin{lem}\label{lem:canonicalOnArbitrarilyLargeFinite}
    Let $\Xi$ be ordered Ramsey, let $\Omega$ be $\omega$-categorical, and let $f \colon \Xi\To \Omega$ be a function. Then for all finite substructures $F\subseteq \Xi$ there is a copy of $F$ in $\Xi$ on which $f$ is canonical.
\end{lem}
\begin{proof}
    Set $n:=n(\Xi)$, and let $m:=o^{\Omega}(n)$. Now $f$ defines a coloring of the $n$-tuples in $\Xi$ by $m$ colors: the color of a tuple $a$ is just the type of $f(a)$ in $\Omega$. 
    Note that if $P$, $S$ are ordered structures, then coloring copies of $P$ in $S$ is the same as coloring tuples of type $\tp(p)$, where $p$ is any tuple which enumerates $P$ -- this is because every copy of $P$ in $S$ contains precisely one tuple of type $\tp(p)$, and every tuple of type $\tp(p)$ in $S$ induces precisely one copy of $P$ in $S$.

    Given any finite substructure $F$ of $\Xi$, enumerate all types of $n$-tuples that occur in $F$ by $t_1,\ldots,t_k$. There is a substructure $S_1$ of $\Xi$ such that whenever all tuples of type $t_1$ in $S_1$ are colored with $m$ colors, then there exists a substructure $H_1$ of $S_1$ isomorphic to $F$ on which the coloring is constant. Further, there is a substructure $S_2$ of $\Xi$ such that whenever all tuples of type $t_2$ in $S_2$ are colored with $m$ colors, then there exists a substructure $H_2$ of $S_2$ isomorphic to $S_1$ on which the coloring is constant. We iterate this $k$ times, arriving at a structure $S_k$. Now going back the argument, we find that $S_k$ contains a copy of $F$ on which all colorings are constant. That means that $f$ is canonical on this copy.
\end{proof}

We remark that this lemma would be false if one dropped the order assumption.

We will now use Lemma~\ref{lem:canonicalOnArbitrarilyLargeFinite} in order to show that for ordered homogeneous Ramsey structures $\Delta$ with finite relational signature, arbitrary functions from $\Delta$ to $\Delta$ \emph{generate} canonical functions from $\Delta$ to $\Delta$. To introduce this notion, we make the following observation. The set $\End(\Delta)$ of endomorphisms of a structure $\Delta$ forms a transformation monoid, i.e., it is closed under composition $f\circ g$ and contains the identity function $\id$. Moreover, it is \emph{closed} (also called \emph{locally closed} or \emph{local}) in the topological sense, i.e., it is a closed subset of the space $D^D$, where $D$ is the domain of $\Delta$ equipped with the discrete topology. This implies that if a set $\F$ of functions from $D$ to $D$ preserves a set of given relations, then so does the smallest closed monoid containing $\F$. This motivates the following definition.

\begin{defn}\label{defn:generatesUnary}
    Let $D$ be a set, $g: D\To D$, and let $\F$ be a set of functions from $D$ to $D$. We say that $\F$ \emph{generates} $g$ iff $g$ is contained in the smallest closed monoid containing $\F$. For a structure $\Delta$ with domain $D$ and a function $f \colon D\To D$, we say that \emph{$f$ generates $g$
    over $\Delta$} iff $\{f\}\cup\Aut(\Delta)$ generates $g$. Equivalently, for every finite subset $F$ of $\Delta$, there exists a term $\alpha_0\circ (f\circ\alpha_1\circ\cdots\circ f\circ \alpha_n)$, where $n\geq 0$ and  $\alpha_i\in\Aut(\Delta)$ for $0\leq i\leq n$, which agrees with $g$ on $F$.
\end{defn}


Note that every operation $f \colon D \To D$ generates
an operation $g$ over $\Delta$ that is canonical 
as a function from $\Delta$ to $\Delta$, namely the
identity operation. 
What we really want is that $f$ generates over $\Delta$ 
a canonical function $g$ which represents $f$ in a certain sense 
-- it should be possible to retain specific properties of $f$ when passing to the canonical function. For example, when $f$ violates a given relation $R$, then we would like to have a canonical $g$ which also violates $R$ -- this is clearly not the case for the identity function. 
Unfortunately, $f$ might be such that it violates a relation $R$
without generating any
 function that is canonical as a function from $\Delta$ to $\Delta$ 
 and that violates $R$. 

We therefore have to refine our method: we would like to fix constants $c_1,\ldots,c_n\in\Delta$ which witness that $f$ violates $R$ and then have canonical behavior relative to these constants, i.e., on the structure $(\Delta,c_1,\ldots,c_n)$ which is $\Delta$ enriched by the constants $c_1,\ldots,c_n$. In order to do this, we must assure that  $(\Delta,c_1,\ldots,c_n)$ still has the Ramsey property. This leads us into topological dynamics.

\subsection{An escapade in topological dynamics}
The goal of this subsection is to show the following proposition by using a recent characterization of the Ramsey property in topological dynamics.

\begin{prop}\label{prop:addingConstantsPreservesRamsey}
    Let $\Delta$ be ordered homogeneous Ramsey, and let $c_1,\ldots,c_n\in \Delta$. Then $(\Delta,c_1,\ldots,c_n)$ is ordered homogeneous Ramsey as well.
\end{prop}

 We remark that it is easy to see that the expansion of any homogeneous structure by finitely many constants is again homogeneous, and that the nontrivial part of the proposition concerns the Ramsey property. We do not know if the same proposition holds if one does not assume $\Delta$ to be ordered.
 
 To prove the proposition, we use a theorem from~\cite{Topo-Dynamics}. A \emph{topological group} is a group $(G;\cdot)$ together with a topology on $G$ such that $(x,y) \mapsto xy^{-1}$ is continuous from $G^2$ to $G$. A group action of $G$ on a topological space
 $X$ is \emph{continuous} iff it is continuous as a function from $G \times X$ into $X$.

\begin{defn}
    A topological group is \emph{extremely amenable} iff any continuous action of the group on a compact Hausdorff space has a fixed point.
\end{defn}

\begin{thm}[Kechris, Pestov, Todorcevic \cite{Topo-Dynamics}]\label{thm:KPT}
    An ordered homogeneous structure is Ramsey iff its automorphism group is extremely amenable.
\end{thm}

Thus the automorphism group of the structure $\Delta$ in Proposition~\ref{prop:addingConstantsPreservesRamsey} is extremely amenable. Note that the automorphism group of $(\Delta,c_1,\ldots,c_n)$ is an open subgroup of $\Aut(\Delta)$. The proposition thus follows from the following fact.

\begin{lem}
    Let $G$ be an extremely amenable group, and let $H$ be an open subgroup of $G$. Then $H$ is extremely amenable.
\end{lem}

\begin{proof}
Let $H$ act continuously on a compact space $X$; we will show that
this action has a fixed point. Denote by $H \backslash G$ the set of
right cosets of $H$ in $G$, i.e. $H \backslash G = \{Hg : g \in G\}$.
Denote by
$\pi \colon G \to H \backslash G$ the quotient map and let $s \colon
H \backslash G \to G$ be a section for $\pi$ (i.e., a mapping
satisfying $\pi \circ s = \mathrm{id}$) such that $s(H) = 1$.
Let $\alpha$ be
the map from $H \backslash G \times G \to H$ defined by
\[ \alpha(w, g) = s(w) g s(wg)^{-1}  \; .\]
For $w \in H \backslash G$ and $g \in G$,
note that $s(w)g$ and $s(wg)$ lie
in the same right coset of $H$, namely $wg$, and hence
the image of $\alpha$ is $H$.
The map $\alpha$ satisfies\footnote{Such maps are called \emph
{cocycles}, and the given identity is called the \emph{cocycle
identity}.}
\begin{align*}
\alpha(w,g_1g_2) & = s(w)g_1 g_2 (s(w g_1 g_2))^{-1} \\
& = s(w) g_1 s(w g_1) s(w g_1)^{-1} g_2 (s(w g_1 g_2))^{-1} \\
& = \alpha(w,g_1) \alpha(wg_1,g_2) \; .
\end{align*}

As $H$ is open, $H \backslash G$ is discrete. Hence, $s$ is
continuous, and therefore $\alpha$ is continuous as a composition of
continuous maps.
The \emph{co-induced action} $G \curvearrowright X^{H \backslash G}$
of $G$ on the product space $X^{H \backslash G}$ is defined by
\[ (g \cdot \xi)(w) = \alpha(w, g) \cdot \xi(wg). \]
To check that this action is continuous, it suffices to see that the
map $(g, \xi) \mapsto (g \cdot \xi)(w)$ is continuous $G \times X^{H
\backslash G} \to X$ for every fixed $w \in H \backslash G$. We
already know that $\alpha$ is continuous and that the action $H
\curvearrowright X$ is continuous. To see that $(g, \xi) \mapsto \xi
(wg)$ is continuous, suppose that $(g_n, \xi_n) \to (g, \xi)$. Let $w=Hk$. As
$g_n \to g$ and $k^{-1}Hk$ is open, we will have that eventually $g_ng^{-1}
\in k^{-1}Hk$, giving that $kg_n (kg)^{-1} \in H$, or, which is the same, $Hkg_n = Hkg$.
We obtain that for sufficiently large $n$, $wg_n = wg$. Therefore $\xi_n(w g_n) \to \xi(wg)$.

By the extreme amenability of $G$, this action has a fixed point $
\xi_0$. Now we check that $\xi_0(H) \in X$ is a fixed point of the
action $H \curvearrowright X$. Indeed, for any $h \in H$, $h \cdot
\xi_0 = \xi_0$ and we have
\[
\xi_0(H) = (h \cdot \xi_0)(H) = \alpha(H, h) \cdot \xi_0(Hh) = h
\cdot \xi_0(H),
\]
finishing the proof.
\end{proof}

\subsection{Minimal unary functions}
Using Proposition~\ref{prop:addingConstantsPreservesRamsey}, we can now prove a `canonisation lemma' that will be central in
what follows. 

\begin{lem}\label{lem:generatesCanonicalWithConstants}
    Let $\Delta$ be ordered homogeneous Ramsey with finite relational signature, $f \colon \Delta \To \Delta$, and let $c_1,\ldots,c_n\in \Delta$. Then $f$ generates over $\Delta$ a function which agrees with $f$ on $\{c_1,\ldots,c_n\}$ and which is canonical as a function from $(\Delta,c_1,\ldots,c_n)$ to $\Delta$.
\end{lem}
\begin{proof}
    Let $(F_i)_{i\in \omega}$ be an increasing sequence of finite substructures of $(\Delta,c_1,\ldots,c_n)$ such that $\bigcup_{i\in\omega} F_i=(\Delta,c_1,\ldots,c_n)$. By Lemma~\ref{lem:canonicalOnArbitrarilyLargeFinite}, for each $i\in\omega$ we find a copy $F_i'$ of $F_i$ in $(\Delta,c_1,\ldots,c_n)$ on which $f$ is canonical. 
    By the homogeneity of $(\Delta,c_1,\ldots,c_n)$, there exist automorphisms $\alpha_i$ of $(\Delta,c_1,\ldots,c_n)$ sending $F_i$ to $F_i'$, for all $i\in\omega$. 
    Since there are only finitely
type conditions for $n((\Delta,c_1,\ldots,c_n))$-tuples, we may assume that if $f$ satisfies a type condition on $F_i'$, then it satisfies the same type condition on $F_{i+1}$.
    Then we can inductively pick automorphisms $\beta_i$ of 
    $(\Delta,c_1,\ldots,c_n)$ such that $\beta_{i+1}\circ f\circ\alpha_{i+1}$ agrees with $\beta_i\circ f\circ\alpha_i$ on $F_i$, for all $i\in\omega$. 
    The union over the functions $\beta_i\circ f\circ \alpha_i \colon F_i\To\Delta$ is a canonical function from $(\Delta,c_1,\ldots,c_n)$ to $\Delta$.
\end{proof}

The set of all closed transformation monoids on a fixed domain $D$ forms a complete lattice with respect to inclusion; it is the lattice of all endomorphism monoids of structures with domain $D$. Lemma~\ref{lem:generatesCanonicalWithConstants} has the following interesting consequence for this lattice.

\begin{defn}
    Let $\N, \M$ be closed monoids over the same domain. We say that $\N$ is \emph{minimal above $\M$} iff $\M\subsetneq \N$ and $\M\subsetneq \R\subseteq \N$ implies $\R=\N$ for all closed monoids $\R$.
\end{defn}

 Clearly, every minimal monoid above $\M$ is generated by a single function together with $\M$; such functions are called \emph{minimal} as well (cf.~\cite{RandomMinOps}).

\begin{lem}\label{lem:minimalMonoidGeneratedByCanonical}
    Let $\Theta$ be a structure with a finite relational signature which is a reduct of an ordered homogeneous Ramsey structure $\Delta$ in a finite relational signature, and let $\N$ be a minimal closed monoid above $\End(\Theta)$. Then there exist constants $c_1,\ldots,c_{n(\Theta)}\in \Delta$ and a function $f$ which is canonical as a function from $(\Delta,c_1,\ldots,c_{n(\Theta)})$ to $\Delta$ such that $\N$ is generated by $\End(\Theta)$ and $f$.
\end{lem}
\begin{proof}
    Pick any $g\in \N\setminus\End(\Theta)$. Since $g\nin\End(\Theta)$, there exist a relation $R$ of $\Theta$ and a tuple $c:=(c_1,\ldots,c_{n(\Theta)})$ such that $R$ is violated on this tuple. By Lemma~\ref{lem:generatesCanonicalWithConstants}, $g$ generates a function $f$ over $\Delta$ which is canonical as a function from $(\Delta,c_1,\ldots,c_{n(\Theta)})$ to $\Delta$ and which is identical with $g$ on $\{c_1,\ldots,c_{n(\Theta)}\}$. Then $f$ and $\End(\Theta)$ generate $\N$.
\end{proof}

\begin{prop}\label{prop:finiteMinimalReducts}
    Let $\Theta$ be a finite relational signature reduct of an ordered homogeneous finite relational signature Ramsey structure $\Delta$. Then there are finitely many minimal closed monoids above $\End(\Theta)$, and every closed monoid containing $\End(\Theta)$
    contains a minimal one. 
\end{prop}
\begin{proof}
    Observe that if $c,d$ are tuples of the same type in $\Delta$, and $f, g$ are canonical functions from $(\Delta,c)$ and $(\Delta,d)$ to $\Delta$, respectively, and their (complete) behaviors are identical, then $f$ and $g$ generate one another over $\Delta$. Thus, there are only finitely many inequivalent (in the sense of `do not generate one another') functions generating minimal monoids. The upper bound for minimal monoids is the following: set $j:=o^\Delta(n(\Theta))$ (there are that many inequivalent choices for the tuple of constants of length $n(\Theta)$ in $\Delta$). For every type of an $n(\Theta)$-tuple $c$ in $\Delta$, set $r_c:=o^{(\Delta,c)}(n(\Delta))$. Set $r$ to be the maximum of the $r_c$. Define moreover $s:=o^{\Delta}(n(\Delta))$. Then a bound for the number of inequivalent minimal functions over $\End(\Theta)$ is $j\cdot s^r$.
\end{proof}

\subsection{Minimal higher arity functions}
Since primitive positive definability is characterized by finitary functions rather than unary functions (recall Theorem~\ref{thm:preservation}), we have to generalize our method to higher arities.

\begin{defn}
    Let $\Xi_1,\ldots,\Xi_m$ be a structures. For a tuple $x$ in the product $\Xi_1\mult\cdots\mult\Xi_m$ and $1\leq i\leq m$, we write
    $x_i$ for the $i$-th coordinate of $x$. The \emph{type} of a sequence of tuples $a^1,\ldots,a^n\in \Xi_1\mult\cdots\mult\Xi_m$, denoted by $\typ(a^1,\ldots,a^n)$, is the $m$-tuple containing the types
    of $(a^1_i,\ldots,a^n_i)$ in $\Xi_i$ for each $1\leq i\leq m$.
\end{defn}

With this definition, also the notions of \emph{type condition}, \emph{behavior}, \emph{($n$-)complete behavior}, and \emph{($n$-)canonical} generalize in complete analogy from functions $f \colon \Xi\To\Omega$, where $\Xi$ is a ``normal'' structure, to functions $f \colon \Xi_1\mult\cdots\mult\Xi_m\To\Omega$ whose domain is a product. It is folklore that the Ramsey property is not lost when going to products; for the reader's convenience, we provide a proof here.

\begin{lem}[The ordered Ramsey product lemma]
\label{lem:ORPL}
    Let $\Xi_1,\ldots,\Xi_m$ be ordered and Ramsey, and set $\Xi:=\Xi_1\mult\cdots\mult\Xi_m$. Let moreover a number $k\geq 1$, an $n$-tuple $(a^1,\ldots,a^n) \in\Xi$, and finite $F_i\subseteq\Xi_i$ be given. Then  there exist finite $S_i\subseteq\Xi_i$ with the property that whenever the $n$-tuples in $S:=S_1\mult\cdots\mult S_m$ of type $\typ(a^1,\ldots,a^n)$ are colored with $k$ colors, then there is a copy of $F:=F_1\mult\cdots\mult F_m$ in $S$ on which the coloring is constant.
\end{lem}

\begin{proof}
    We use induction over $m$. The induction beginning $m=1$ is trivial, so assume $m>1$ and that the lemma holds for $m-1$.
    For all $1\leq i\leq n$, set $c^i:=(a^i_1,\ldots,a^i_{m-1})$. By the induction hypothesis, there exist finite $S_i\subseteq\Xi_i$ for all $1\leq i\leq m-1$ such that whenever its $n$-tuples of type $\typ(c^1,\ldots,c^n)$ are colored with $k$ colors, then there is a copy of $F_1\mult\cdots\mult F_{m-1}$ in $S_1\mult\cdots\mult S_{m-1}$ on which the coloring is constant. Let $p$ be the number of $n$-tuples of this type in $S_1\mult\cdots\mult S_{m-1}$. Also by induction hypothesis, there exists a finite $S_{m,1}\subseteq\Xi_m$ with the property that whenever its $n$-tuples of type $\typ(a^1_m,\ldots,a^n_m)$ are colored with $k$ colors, then it contains a monochromatic copy of $F_m$. Further, there is a finite $S_{m,2}\subseteq \Xi_m$ with the property that whenever its subsets of this type are colored with $k$ colors, then it contains a monochromatic copy of $S_{m,1}$. Continue constructing finite substructures of $\Xi_m$ like that, arriving at $S_m:=S_{m,p}$.

    We claim that $S:=S_1\mult\cdots\mult S_m$ has the desired property. To see this, let a coloring $\chi$ of the $n$-tuples in $S$ of type $\typ(a^1,\ldots,a^n)$  be given. Let $b(1),\ldots,b(p)$ be an enumeration of all the $n$-tuples in $S_1\mult\cdots\mult S_{m-1}$ which have type $\typ(c^1,\ldots,c^n)$. For $1\leq i\leq p$ and $1\leq j\leq n$, we write $b(i)^j$ for the $j$-th component of $b(i)$ (note that this component is an $(m-1)$-tuple in $S_1\mult\cdots\mult S_{m-1}$). Now for all $1\leq i\leq p$, define a coloring $\chi^i$ of the $n$-tuples $t=(t^1,\ldots,t^n)$ in $S_m$ of type $\typ(a^1_m,\ldots,a^n_m)$ by setting $\chi^i(t):=\chi(b(i)^1*t^1,\ldots,b(i)^n*t^n)$, where $r*s$ denotes the concatenation of two tuples $r,s$. By thinning out $S_m$ $p$ times, we obtain a copy $F_m'$ of $F_m$ in $S_m$ on which each coloring $\chi^i$ is constant with color $q^i$. Now by that construction, all $n$-tuples $b(i)$ have been assigned a color $q^i$, the assignment thus being a coloring of all the $n$-tuples of type $\typ(c^1,\ldots,c^n)$ in $S_1\mult\cdots\mult S_{m-1}$. By the choice of that product, there is a copy $F'_1\mult\cdots\mult F'_{m-1}$ of $F_1\mult\cdots\mult F_{m-1}$ in $S_1\mult\cdots\mult S_{m-1}$ on which that coloring is constant, say with value $q$. But that means that if a tuple $(d^1,\ldots,d^n) \in F_1'\mult\cdots\mult F_m'$ has type $\typ(a^1,\ldots,a^n)$, then $\chi(d^1,\ldots,d^n)=q$, proving our statement.
\end{proof}

We now generalize the notion of a transformation monoid to higher arities. Denote the set of all polymorphisms of $\Delta$ by $\Pol(\Delta)$. Irrespectively of the structure $\Delta$, this set contains all finitary projections and is closed under composition. Sets of finitary functions with these two properties are referred to as \emph{clones} -- for a survey of clones on infinite sets, see~\cite{GoldsternPinsker}. In addition, the clone $\Pol(\Delta)$ is a closed subset of the sum space of the spaces $D^{D^n}$, where $D$ is again taken to be discrete; such clones are called \emph{closed}, \emph{local}, or \emph{locally closed} (cf. the corresponding terminology for monoids before). This means that if a set $\F$ of finitary functions on a domain $D$ preserves a set of given relations, then so does the smallest closed clone containing $\F$, motivating the following extension of Definition~\ref{defn:generatesUnary}.

\begin{defn}\label{defn:generatesFinitary}
    Let $D$ be a set, $g: D^m\To D$, and let $\F$ be a set of finitary operations on $D$. We say that $\F$ \emph{generates} $g$ iff $g$ is contained in the smallest closed clone containing $\F$. For a structure $\Delta$ with domain $D$ and a function $f \colon D^n\To D$, we say that \emph{$f$ generates $g$
    over $\Delta$} iff $\{f\}\cup\Aut(\Delta)$ generates $g$. Equivalently, for every finite subset $F$ of $\Delta^m$, there exists an $m$-ary term built from $f$, $\Aut(\Delta)$, and projections, which agrees with $g$ on $F$.
\end{defn}

As before, finitary functions on ordered homogeneous Ramsey structures generate canonical functions, and we can add constants to the language.

\begin{lem}\label{lem:canonicalConstantsHigherArity}
    Let $\Delta$ be ordered homogeneous Ramsey with finite relational signature, and let $f \colon \Delta^m\To \Delta$. Let moreover finite tuples $c_1=(c_1^1,\ldots,c_1^{n_1}),\ldots,c_m=(c_m^1,\ldots,c_m^{n_m})$ of constants in $\Delta$ be given. Then $f$ generates over $\Delta$ an $m$-ary operation $g$ on $\Delta$ which is canonical as a function from $(\Delta,c_1)\mult\cdots\mult(\Delta,c_m)$ to $\Delta$ and which agrees with $f$ on all tuples $(c_1^{j_1},\ldots,c_m^{j_m})$.
\end{lem}
\begin{proof}
    We recommend combining Lemma~\ref{lem:ORPL} with the methods of the preceding section in order to prove this.
\end{proof}

The set of all closed clones on a fixed domain $D$ forms a complete lattice with respect to inclusion; it is the lattice of all polymorphism clones of structures with domain $D$. This lattice has been investigated in universal algebra (see~\cite{Pin-morelocal}).

\begin{defn}\label{defn:minimalClone}
    For closed clones $\C, \D$ on the same set, we say that $\D$ is \emph{minimal} above $\C$ iff $\C\subsetneq \D$ and $\C\subsetneq \E\subseteq \D$ implies $\E=\D$ for all closed clones $\E$. Every minimal closed clone above $\C$ is generated by $\C$ plus a single function $f$ outside $\C$; we call such a function $f$ \emph{minimal} above $\C$ if there is no function of smaller arity which generates (together with $\C$) the same closed clone as $f$.
\end{defn}

Lemma~\ref{lem:canonicalConstantsHigherArity} allows us to find the minimal clones above a closed clone on an ordered  homogeneous Ramsey structure. The main difference here compared with monoids is that the arities of minimal canonical functions are not bounded a priori, which means that there could be infinitely many minimal clones. The following lemma, which has been observed in~\cite{tcsps-journal}, yields a bound on the arities of minimal functions.

\begin{lem}\label{lem:arityReduction}
    Let $\Theta$ be a structure, $m \geq 1$, and let $R\subseteq \Theta^n$ be a relation which intersects precisely $m$ $n$-orbits of $\Theta$. If a function
    $f \colon \Theta^p \To \Theta$ violates $R$, then $f$ generates over $\Theta$ a function of arity $m$ which violates $R$, too.
\end{lem}
\begin{proof}
Let $O_1,\ldots,O_m$ be the orbits of $\Theta$ that are intersect $R$, and fix arbitrary tuples $s_i\in O_i$. Since $f$ violates $R$, there exist $r_1,\ldots,r_p\in R$ such that $f(r_1,\ldots,r_p)\nin R$. Say that $b_i\in O_{j_i}$, for all $1\leq i\leq p$, and choose for all $1\leq i\leq p$ an automorphism $\alpha_i$ of $\Theta$ sending $s_{j_i}$ to $r_{i}$. The function $g(x_1,\ldots,x_m):=f(\alpha_{1}(x_{i_1}),\ldots,\alpha_p(x_{i_p}))$ has arity $m$ and violates $R$ since $g(s_1,\ldots,s_m)=f(r_1,\ldots,r_p)$ is not in $R$.
\end{proof}


\begin{prop}\label{prop:finiteMinimalClones}
    Let $\Theta$ be a finite relational signature reduct of an ordered homogeneous Ramsey structure $\Delta$ with finite relational signature. Then there are finitely many minimal closed clones above $\Pol(\Theta)$, and every closed clone containing $\Pol(\Theta)$
    contains a minimal one. 
\end{prop}
\begin{proof}
    Let $R_1,\ldots,R_n$ be the relations of $\Theta$. If $f$ is a minimal operation above $\Pol(\Theta)$, then it violates a relation $R_i$. By Lemma~\ref{lem:arityReduction}, it generates over $\Theta$ a function of arity $o^\Theta(k_i)$, where $k_i$ is the arity of $R_i$, which still violates $R_i$. Setting $m$ to be the maximum of the $o^\Theta(k_i)$ where $1\leq i\leq n$, we get that every minimal clone above $\Pol(\Theta)$ is generated by a function of arity at most $m$. By Lemma~\ref{lem:canonicalConstantsHigherArity}, such functions can be made canonical -- the rest of the proof is just like the proof of Proposition~\ref{prop:finiteMinimalReducts}.
\end{proof}

If one wishes to determine the minimal clones above the endomorphism monoid of a structure $\Theta$, then there is a bound on the arities of minimal functions which only depends of the number of $2$-orbits of the structure $\Theta$, rather than the number of orbits of possibly longer tuples as in the preceding proof. 

\begin{defn}
    Let $D$ be a set, and let $f \colon D^m \To D$ be an operation on $D$. Then $f$ is called \emph{essentially unary} iff there exist $1\leq i\leq m$ and $F \colon D\To D$ such that $f(x_1,\ldots,x_m)=F(x_i)$. Conversely, $f$ is called \emph{essential} iff it is not essentially unary.
\end{defn}

\begin{prop}\label{prop:finiteMinimalClonesAboveEnd}
    Let $\Theta$ be any relational structure for which $o^\Theta(2)$ is finite. Then every minimal closed clone above $\End(\Theta)$ is generated by a function of arity at most $2\cdot o^\Theta(2)-1$ together with $\End(\Theta)$.
\end{prop}
\begin{proof}
    Let $\D$ be a minimal closed clone above $\End(\Theta)$. If all the functions in $\D$ are essentially unary, then $\D$ is generated by a unary operation together with $\End(\Theta)$ and we are done. Otherwise, let $f$ be an essential operation in $\D$. Then one can verify that $f$ violates the $3$-ary relation $P_3$ defined by the formula $(x=y) \vee (y=z)$. The assertion then follows from Lemma~\ref{lem:arityReduction}: the $3$-ary subrelation of $P_3$ defined by the formula $x=y$ clearly consists of $o^\Theta(2)$ orbits in $\Theta$; similarly, the $3$-ary subrelation defined by $y=z$ consists of the same number of orbits. Since $P_3$ is the union of these two subrelations, and since the intersection of the two subrelations consists of exactly one orbit (namely, the triples with three equal entries), we obtain $2\cdot o^\Theta(2)-1$ different orbits for tuples in $P_3$.
\end{proof}

Observe that in Proposition~\ref{prop:finiteMinimalClonesAboveEnd}, if $\Theta$ is a reduct of a structure $\Delta$, we can also write $2\cdot o^\Delta(2)-1$ for the arity bound if we wish to have a bound which is independent of $\Theta$, since $\Delta$ has at least as many $2$-orbits as $\Theta$.

\section{The Algorithm}\label{sect:algorithm}

We now present the algorithm proving Theorem~\ref{thm:main:pp}; the proof of the two statements of Theorem~\ref{thm:main:exep} is a subset. 
So we are given formulas $\phi_0,\ldots,\phi_n$ over $\Gamma$ which define relations $R_0,\ldots,R_n$ on the domain $D$ of $\Gamma$. 
Set $\Theta$ to be the reduct $(D;R_1,\ldots,R_n)$ of $\Gamma$, and write $R:=R_0$. 
We will decide whether there is a primitive positive definition of $R$ in $\Theta$.

\subsection{Operationalization}
If there is no such definition, then since $\Theta$ is $\omega$-categorical, by Theorem~\ref{thm:preservation} there is a polymorphism $f$ of $\Theta$ which violates $R$; we call $f$ a \emph{witness}. Our algorithm will now try to build a witness. If it fails to do so, then $R$ is primitive positive definable in $\Theta$; otherwise, it is not.

\subsection{Arity reduction}
Let $k$ be the arity of $R$. By Lemma~\ref{lem:arityReduction}, if there exists a witness, then there exists also a witness of arity equal to the number of those $k$-orbits in $\Theta$ that intersect $R$. This number is not larger than $o^\Theta(k)$, which is not larger than $o^\Gamma(k)$ since $\Aut(\Gamma)\subseteq\Aut(\Theta)$. Set $m:=o^\Gamma(k)$; the algorithm now tries to detect a witness of arity $m$.

\subsection{Ramseyfication}
If $f$ is a witness of arity $m$, then there are $k$-tuples $c_1,\ldots,c_m\in R$ such that $f(c_1,\ldots,c_m)\nin R$. 
By assumption, $\Gamma$ has a first-order definition in an ordered homogeneous structure $\Delta$ that is finitely bounded, Ramsey, and has finite relational signature. 
By Lemma~\ref{lem:canonicalConstantsHigherArity}, $f$ generates over $\Delta$ an $m$-ary function $g$ which is canonical as a function from 
$(\Delta,c_1)\mult\cdots\mult (\Delta,c_m)$ to $\Delta$ and which agrees with $f$ on all $m$-tuples whose $i$-th component is taken from the $k$-tuple $c_i$ for all $1\leq i\leq m$. In particular, $g$ still violates $R$ and preserves $\Theta$, and hence is a witness, too. Our algorithm thus tries to find a witness of this form.

\subsection{Finite representation}
Let $n:=\max(s,n(\Delta),3)$, where $s$ is the maximal size of the finitely many finite forbidden substructures of $\Delta$. Since $n \geq n(\Delta)$, a function from $(\Delta,c_1)\mult\cdots\mult (\Delta,c_m)$ to $\Delta$ is canonical iff it is $n$-canonical. Such functions can thus be represented as functions from $S^{(\Delta,c_1)}_n\mult\cdots\mult S^{(\Delta,c_m)}_n$ to $S^\Delta_n$. Note that the type space $S^{(\Delta,c_i)}_n$ only depends on the type of $c_i$ in $\Delta$. In other words, if we replace the tuple $c_i$ by a tuple $d_i$ of the same type in $\Delta$, we obtain the same possibilities of complete behavior. Since $o^\Delta(k)$ is finite, there are only finitely many choices of types for each $c_i$ -- our algorithm tries all such choices (since $\Delta$ has a finite relational signature, and is homogeneous, those choices can be made effectively). 
For each choice for the types of the $c_i$, and for each function $\sigma$ from $S^{(\Delta,c_1)}_n\mult\cdots\mult S^{(\Delta,c_m)}_n$ to $S^\Delta_n$, the algorithm checks whether $\sigma$ is the behavior of a witness.
\subsection{Verification}
Given $\sigma$, we verify the following.
\begin{itemize}
        \item (Compatibility.) If $\sigma$ is a behavior of a canonical operation, then for all $1\leq k\leq n$ it must also be extendible to a function from $S^{(\Delta,c_1)}_k\mult\cdots\mult S^{(\Delta,c_m)}_k$ to $S^\Delta_k$. This is possible in the following situation: if $s$ is an $n$-type, then it has certain \emph{$k$-subtypes} $t$, i.e., projections of tuples of type $s$ onto $k$ coordinates satisfy $t$. Now products of $k$-subtypes are automatically sent to a $k$-subtype under $\sigma$: if $s_1,\ldots,s_m$ are $n$-types and $I\subseteq \{1,\ldots,n\}$ is a set of size $k$ inducing $k$-subtypes $t_i$ of $s_i$, then $I$ induces a $k$-subtype of $\sigma(s_1,\ldots,s_m)$.
             Our algorithm checks for $n$-types $p_1,q_1,\ldots,p_m,q_m$ and all $I, J\subseteq \{1,\ldots,n\}$ that if $I$ and $J$ induce  identical $k$-subtypes in $p_i$ and $q_i$, respectively, then they induce identical $k$-subtypes in $\sigma(p_1,\ldots,p_m)$ and $\sigma(q_1,\ldots,q_m)$ -- otherwise, $\sigma$ is rejected as a candidate. If on the other hand $\sigma$ satisfies this condition, then it naturally extends to a function from $S^{(\Delta,c_1)}\mult\cdots\mult S^{(\Delta,c_m)}$ to $S^\Delta$ respecting arities, and we can compute the value of this function for every argument. In the following, we write $\sigma$ for this extended function.

        \item (Violation.) Since $R$ has a first-order definition in $\Delta$, and automorphisms of $\Delta$ preserve first-order formulas, it follows that $R$ is a union of orbits, i.e., if $a,b$ are of the same type, then $a\in R$ iff $b\in R$.
                Set $t:=\sigma(\tp^{(\Delta,c_1)}(c_1),\ldots,\tp^{(\Delta,c_m)}(c_m))$. Our algorithm checks that $t$ is not a type in $R$, since we only want to accept $\sigma$ if it is the behavior of an operation which violates $R$ on $c_1,\ldots,c_m$.

        \item (Preservation.) For every relation $R_i$ from $\Theta$, we check that $\sigma$ ``preserves'' $\Theta$ as follows: write $p$ for the arity of $R_i$. For all $p$-types $t_1,\ldots,t_m$ of tuples in $R_i$, we verify that $\sigma(t_1,\ldots,t_m)$ is the type of a tuple in $R_i$; otherwise we reject $\sigma$.

\end{itemize}

We now argue that the algorithm finds a $\sigma$ satisfying our three conditions if and only if there is an $m$-ary polymorphism of $\Theta$ that violates $R$. It is clear that the type function of a witness will satisfy all the conditions, so one direction is straightforward.
For the opposite direction, suppose that $\sigma$ is accepted by our algorithm. We build a canonical operation from $(\Delta,c_1)\mult\cdots\mult (\Delta,c_m)$ to $\Delta$ in three steps. Let $\tau$ be the signature of $\Delta$.

We first construct an infinite structure $\Pi$ with domain $D^m$
and signature $\tau \cup \{\sim\}$, where $\sim$ is a new binary relation symbol,
as follows. This relation is for the proper treatment of equality of function values
when realizing the behavior $\sigma$. 
For all $(a_1,b_1),\dots,(a_m,b_m) \in D^2$ with types $t_1,\dots,t_m$ 
in $(\Delta,c_1),\dots,(\Delta,c_m)$, respectively, if the 2-type $\sigma(t_1,\dots,t_m)$ contains $x_1 = x_2$ then we set $(a_1,\dots,a_m) \sim (b_1,\dots,b_m)$. 
Note that since $n \geq 3$ and because of the compatibility constraints and transitivity of equality, $\sim$ then denotes an equivalence relation on $D^m$. 
The other relations of $\Pi$ are defined as follows. Let $R$ be a $k$-ary relation from
$\tau$. 
We add the $k$-tuple $((a^1_1,\dots,a^1_m),\dots,(a^k_1,\dots,a^k_m))$ to the relation $R$ of $\Pi$ if and only if $R(x_1,\dots,x_k)$ is contained in $\sigma(t_1,\dots,t_m)$,
where $t_i$ is the type of the tuple $(a^1_i,\dots,a^k_i) \in D^k$ in $(\Delta,c_i)$.
 Since $n \geq n(\Delta) \geq k$, 
this is well-defined by the compatibility item of our algorithm.

The quotient structure $\Pi /_{\sim}$ is defined to be the $\tau$-structure 
whose domain is the set $D/_{\sim}$ of all equivalence classes of $\sim$, and
where $R(E_1,\dots,E_p)$ holds for a $p$-ary $R \in \tau$ and $E_1,\dots,E_p \in D/_{\sim}$ if and only if there are $b_1 \in E_1,\dots,b_p \in E_p$ such that
$R(b_1,\dots,b_p)$ holds in $\Pi$. 
The final step is to show that there exists an embedding $f$ of $\Pi/_{\sim}$ into $\Delta$.
By $\omega$-categoricity of $\Delta$ and a standard compactness argument (see, e.g., Lemma 2 in~\cite{BodDalJournal}),
it suffices to show every finite substructure $\Omega$ of $\Pi/_{\sim}$ embeds into $\Delta$. 
This follows from the fact that none of the forbidden substructures embeds into $\Omega$,
since $n \geq s$, where $s$ is the size of the largest obstruction.

Finally, observe that the mapping $g$ from $D^m$ to $D$ that maps every $u$ in $D^m$ to
$f({u}/_{\sim})$ (where ${u}/_{\sim}$ denotes the $\sim$-equivalence class of $u$ in $\Pi$)
is a polymorphism of $\Theta$ by the preservation item of the algorithm,
and that $g$ violates $R$ by the violation item of the algorithm. 

\section{Decidability of Polymorphism Conditions}
In all known cases of structures $\Gamma$ with a finite relational signature and a 
first-order definition in a finitely
bounded ordered homogeneous Ramsey structure, $\Csp(\Gamma)$ is tractable
if and only if there exists a $4$-ary polymorphism $f$ of $\Gamma$
and an automorphism $\alpha$ of $\Gamma$ such that for all elements $x,y,z$
of $\Gamma$
\begin{align*} f(x,y,z,z) & = \alpha(f(y,z,x,y)) & (*)
\end{align*}
One can show that condition $(*)$ describes indeed the frontier between tractability
and NP-hardness for reducts of $({\mathbb Q};<)$ and the random graph.
It has also been conjectured to be the tractability frontier of $\Csp(\Gamma)$ for structures $\Gamma$ with a finite domain~\cite{JBK,Siggers}.

When $\Gamma$ is given by defining quantifier-free formulas over $\Delta$, and $\Delta$
is given by its forbidden induced substructures, then 
 the existence of $f,\alpha$ satisfying condition $(*)$ can be tested by an algorithm, 
 by the techniques developed here.
A $4$-ary operation $f$ satisfies this condition 
if and only if the type function $\sigma$ of $f$ satisfies $\sigma(t_1,t_2,t_3,t_3)=\sigma(t_2,t_3,t_1,t_2)$ for all $n$-types $t_1,t_2,t_3$ of $\Delta$, 
where $n \geq \max(n(\Delta),3,s)$ and $s$ is the maximal obstruction size of 
$\Delta$.

\section{Discussion and Open Problems}
We presented an algorithm that decides primitive positive 
definability in finite relational signature reducts $\Gamma$ of structures that are ordered, Ramsey, homogeneous, finitely bounded, and with finite relational signature.  
All of those structures $\Gamma$ are $\omega$-categorical.
While the condition for $\Gamma$ might appear rather restrictive at first sight, it is actually are quite general: we want
to point out that we do not require that $\Gamma$ is Ramsey, we only require that 
$\Gamma$ is definable in a Ramsey structure. We do not know 
of a single homogeneous structure $\Gamma$ with finite relational signature which 
is \emph{not} the reduct of an ordered homogeneous Ramsey structure with finite relational signature. 

\begin{problem}
Does every structure which is homogeneous in a finite relational signature have
a homogeneous expansion by finitely many relations such that the
resulting structure is Ramsey?
\end{problem}

A variant of this problem is the following.

\begin{problem}
Does every $\omega$-categorical structure have an
$\omega$-categorical expansion which is Ramsey?
\end{problem}



Note that our method is non-constructive: the algorithm does not produce a primitive
positive definition in case that there is one. It is an interesting open problem to come up with bounds on the number of
existential variables that suffice for a primitive positive definition of $R$ in $\Theta$. For many structures $\Gamma$ of practical interest, such as $({\mathbb Q};<)$ or
the random graph, our algorithm can certainly be tuned so that $\exprpp(\Gamma)$ becomes feasible for reasonable input size; in particular, the
gigantic Ramsey constants involved in the proofs of our results do not
affect the running time of our procedure.

Another important open problem is whether the method can be extended to show
decidability of our computational problem for \emph{first-order} definability instead of primitive positive, existential positive, and existential definability; we denote this
computational problem by $\exprpp(\Gamma)$. 
By the theorem of Ryll-Nardzewski, first-order definability
is characterized by preservation under automorphisms, i.e., surjective self-embeddings.
But the requirement of surjectivity is difficult to deal with in our approach.

\begin{problem}
Let $\Delta$ be a structure which is ordered, homogeneous, Ramsey, finitely bounded, and has a finite relational signature, and let $\Gamma$ be a reduct of $\Delta$ with finite relational signature. 
Is the problem $\exprfo(\Gamma)$ decidable?
\end{problem}

\bibliographystyle{plain}
\bibliography{local}

\end{document}